\documentclass[a4paper,12pt]{article}

\usepackage{amssymb}
\usepackage{latexsym}
\usepackage{amsfonts}
\usepackage{amsmath}
\usepackage{amsthm}
\usepackage{hyperref}
\usepackage{tensor}
\usepackage{mathtools}
\usepackage{mdwlist}
\usepackage{pgf,tikz,pgfplots}
\pgfplotsset{width=7.5cm,compat=1.9}

\usetikzlibrary{arrows,patterns,positioning}
\usetikzlibrary{calc}

\usepackage[margin=1in]{geometry}

\usepackage[compact]{titlesec}

\usepackage[scale=0.9]{tgheros}
\usepackage[T1]{fontenc}

\DeclareMathOperator{\SL}{SL}

\DeclareMathOperator{\GL}{GL}

\renewcommand{\leq}{\leqslant}
\renewcommand{\geq}{\geqslant}

\newcommand{\F}{\mathbb F}
\newcommand{\K}{\mathbb K}

\newcommand{\B}{\mathcal B}

\newcommand{\LL}{\mathcal L}

\newcommand{\V}{\mathcal V}

\theoremstyle{plain}
\newtheorem{lemma}{Lemma}
\newtheorem{theorem}[lemma]{Theorem}

\theoremstyle{definition}
\newtheorem{definition}[lemma]{Definition}

\newtheorem{question}[lemma]{Question}

\numberwithin{equation}{section}
\numberwithin{lemma}{section}

\begin{document}

\title{An explicit construction for large sets of infinite dimensional $q$-Steiner systems}
\author{Daniel R. Hawtin\footnote{Address: Faculty of Mathematics, University of Rijeka, Rijeka 51000, Croatia. \href{mailto:dhawtin@math.uniri.hr}{dhawtin@math.uniri.hr}}
}

\maketitle
\begin{abstract}
 Let $V$ be a vector space over the finite field $\F_q$. A \emph{$q$-Steiner system}, or an $S(t,k,V)_q$, is a collection $\B$ of $k$-dimensional subspaces of $V$ such that every $t$-dimensional subspace of $V$ is contained in a unique element of $\B$. A \emph{large set} of $q$-Steiner systems, or an $LS(t,k,V)_q$, is a partition of the $k$-dimensional subspaces of $V$ into $S(t,k,V)_q$ systems. In the case that $V$ has infinite dimension, the existence of an $LS(t,k,V)_q$ for all finite $t,k$ with $1<t<k$ was shown by Cameron in 1995. This paper provides an explicit construction of an $LS(t,t+1,V)_q$ for all prime powers $q$, all positive integers $t$, and where $V$ has countably infinite dimension. 
\end{abstract}

\section{Introduction}\label{sect:intro}

Let $\V$ be a set. A \emph{Steiner system}, or more precisely an $S(t,k,\V)$, is a collection $\B$ of $k$-subsets of $\V$ such that every $t$-subset of $\V$ is contained in precisely one element of $\B$. There is a vast literature on Steiner systems in the case that $\V$ is finite, and these have many applications in combinatorics and other areas; we refer the reader to \cite{colbourn2006steiner} for basic properties and results. Keevash \cite{keevash2014existence} proved that Steiner systems exist whenever their parameters satisfy certain basic divisibility conditions, provided $|\V|$ is large enough. Despite this result, there are currently no known Steiner systems for $\V$ finite and $t\geq 6$, to the best knowledge of this author. In the case that $\V$ is infinite, K{\"o}hler \cite{kohler1977unendliche} proved that $S(t,k,\V)$ systems exist for all finite $t,k$ with $1< t<k$ and Grannell \emph{et al.}~\cite{grannell1987countably,grannell1991infinite} explicitly constructed $S(t,t+1,\V)$ systems.

A large set of Steiner systems, or an $LS(t,k,\V)$, is a partition the set of all $k$-subsets of $\V$ into $S(t,k,\V)$ systems. The existence of an $LS(2,3,\V)$ for all $n=|\V|$ with $n\neq 7$ and $n\equiv 1,3 {\pmod 6}$ was established in a series of papers by Lu \cite{lu1983large}, that was unfortunately never completely published; a full proof can be found in \cite{teirlinck1991completion}. The only other known large sets of Steiner systems, for $\V$ finite, are $LS(2,4,\V)$ systems with $|\V|=13$ \cite{chouinard1983partitions} and $|\V|=16$ \cite{mathon1997searching}. In the infinite case, the explicit construction of $S(t,t+1,\V)$ systems by Grannell \emph{et al.}~\cite{grannell1991infinite}, mentioned in the previous paragraph, in fact produces an $LS(t,t+1,\V)$, and Cameron \cite{cameron1995note} proved that an $LS(t,k,\V)$ exists for all $t,k$ with $1<t<k$ (note that the size of $\V$ is not restricted to be countably infinite in either of these results). 

Here we are concerned with the $q$-analogues of Steiner systems. Let $V$ be a vector space over $\F_q$. A \emph{$q$-Steiner system}, or an $S(t,k,V)_q$, is a collection $\B$ of $k$-dimensional subspaces of $V$ such that each $t$-dimensional subspace of $V$ is contained in precisely one element of $\B$. Ray-Chaudhuri and Singhi \cite{ray1989q} proved an asymptotic existence result for $q$-Steiner systems as far back as 1989. More recently, Braun~{\em et al.}~\cite{braunqSteinerexist2016} constructed the first examples, producing a large number of non-isomorhpic $S(2,3,\F_2^{13})_2$ systems, which are the only currently know $q$-Steiner systems. As far as the author is aware, there are no known explicit constructions in the case that $V$ has infinite dimension.

A large set of $q$-Steiner systems, or an $LS(t,k,V)_q$, is a partition of the set of all $k$-subsets of $V$ into $S(t,k,V)_q$ systems. Cameron also considers these vector space analogues in \cite{cameron1995note}, proving that $LS(t,k,V)_q$ systems exist for all finite $t,k$ with $1<t<k$. (Note that Cameron's result also includes the case for vector spaces over infinite fields.) Again, as far as this author is aware, there are no known explicit constructions of large sets of $q$-Steiner systems; the main result of this paper, below, is to provide such a construction.

\begin{theorem}\label{thm:mainresult}
 Let $q$ be a prime power, let $t$ be a positive integer, let $\K$ be the algebraic closure of $\F_q$ and let $\LL$ be as in Definition~\ref{def:largeset}. Then $\LL$ is an $LS(t,t+1,\K)_q$.
\end{theorem}

The paper is organised as follows: Section~\ref{sect:prelim} introduces some notation and preliminary results and Section~\ref{sect:proof} gives the construction and proof of Theorem~\ref{thm:mainresult}.

\section{Preliminaries}\label{sect:prelim}

Let $\K$ be the algebraic closure of $\F_q$. Then $\K=\bigcup_{i=1}^\infty \F_{q^i}$ has countably infinite dimension as a vector space over $\F_q$. All vector spaces considered in this paper will be $\F_q$-subspaces of $\K$.

A \emph{linearised polynomial} or a \emph{$q$-polynomial} is a polynomial of the form
\[
 f(x)=a_0x + a_1x^q + \cdots + a_{k-1}x^{q^{k-1}} + a_kx^{q^k},
\]
where $k\geq 0$ and $a_i\in\K$ for each $i=0,1,\ldots,k$ (see, for instance, \cite[Chapter~3, Section~4]{lidl1997finite}). The roots of $f$ form a vector space $V$ (see \cite[Theorem~3.50]{lidl1997finite}) with each root having multiplicity a power of $q$ (possibly $1=q^0$). The roots of $f$ have multiplicity $1$ if and only if $a_0\neq 0$. If $a_k\neq 0$ then we say the \emph{$q$-degree} of $f$ is $k$ and, if $a_0\neq 0$ as well, then the roots of $f$ form a $k$-dimensional subspace of $\K$. We denote the vector space of roots of $f$ by $\ker f$.

Suppose $f$ is as above and that $a_0,a_k\neq 0$. The coefficients of $f$ are related to elements $v_1,\ldots,v_k\in \K$ consisting of a basis for the subspace $V$ of roots of $f$ via the determinant of the $(k+1)\times(k+1)$ matrix
\begin{equation}\label{eq:matrix}
 M=
 \begin{bmatrix}
  x & x^q & x^{q^2} & \cdots & x^{q^k}\\
  v_1 & v_1^q & v_1^{q^2} & \cdots & v_1^{q^k} \\
  v_2 & v_2^q & v_2^{q^2} & \cdots & v_2^{q^k} \\
  \vdots & \vdots & \vdots & \ddots & \vdots\\
  v_k & v_k^q & v_k^{q^2} & \cdots & v_k^{q^k}
 \end{bmatrix}.
\end{equation}
In particular, expanding along the top row of $M$ shows that (up to replacing $f$ by $\alpha f$ for some $\alpha\in\K$) the coefficient of $x^{q^i}$ is given by the determinant of the $k\times k$ minor of $M$ obtained by omitting the top row and the column containing $x^{q^i}$. The minors of $M$ just discussed are related to the Dickson invariants, see \cite{dickson1911fundamental} and \cite[Section~8.1]{benson1993polynomial}. Note that if $A\in\SL_k(q)$ then $V$ and the determinant of $M$ (and each of the relevant $k\times k$ minors) is invariant under the map
\[
 \begin{bmatrix}
  v_1  \\
  v_2  \\
  \vdots \\
  v_k
 \end{bmatrix}
 \mapsto
 A
 \begin{bmatrix}
  v_1  \\
  v_2  \\
  \vdots \\
  v_k
 \end{bmatrix}.  
\]
If instead $A$ above is an element of $\GL_k(q)$ then while $V$ is still preserved by the above map, the determinant of $M$, and the minors considered above, are all scaled by $\det(A)\in \F_q$.

\section{The construction}\label{sect:proof}

Throughout this section, let $\K$ be the algebraic closure of $\F_q$ and fix a positive integer $t$. Recall that all vector spaces are vector spaces over $\F_q$. The construction is defined below.

For each $a\in \K\setminus\{0\}$ and $B=(b_1,\ldots,b_t)\in \K^t$ define $f_{a,B}$ to be the following linearised polynomial in $\K[x]$:
\begin{equation}\label{eqn:faB}
 f_{a,B}=a^q x+b_1 x^q\cdots+b_t x^{q^t}+(-1)^{t+1} a x^{q^{t+1}}.
\end{equation}
Note that the relationship between the coefficients of $x$ and $x^{q^{t+1}}$ in $f_{a,B}$ comes from considering the expansion for the determinant of $M$ in (\ref{eq:matrix}). Further, for each $a\in \K\setminus\{0\}$, define a set $\B_a$ of subspaces of $\K$:
\begin{equation}\label{eqn:Bsuba}
 \B_a=\{\ker f_{a,B}\mid B\in \K^t\}.
\end{equation}
Note that each element of $\B_a$ is a $(t+1)$-space, since the coefficients of $x$ and $x^{q^{t+1}}$ in $f_{a,B}$ are non-zero when $a\neq 0$.
 
\begin{definition}\label{def:largeset}
 Using the notation as in (\ref{eqn:faB}) and~(\ref{eqn:Bsuba}), define
 \[
  \LL=\{\B_a\mid a\in\K\setminus\{0\}\}.
 \] 
\end{definition}



Before moving towards the proof of the main theorem, the next result effectively shows that $\B_a=\B_b$ whenever $a$ and $b$ lie in the same one-dimensional subspace of $\K$.

\begin{lemma}\label{lem:scalarMult}
 Let $a\in \K\setminus \{0\}$, let $B=(b_1,\ldots,b_t)\in \K^t$, let $f_{a,B}$ be as in (\ref{eqn:faB}) and let $V=\ker f_{a,B}$. Then $V=\ker f_{c,D}$ for some $c\in \K\setminus\{0\}$ and $D\in \K^t$ if and only if $c=\lambda a$ and $D=\lambda B$ for some $\lambda\in\F_q$.
\end{lemma}

\begin{proof}
 Suppose there exists some $c\in \K\setminus\{0\}$ and $D\in \K^t$ such that $\ker f_{c,D}=V$. Since the coefficients of $x$ in each of $f_{a,B}$ and $f_{c,D}$ are non-zero, it follows that the roots of each all have multiplicity $1$. The fact that both polynomials have the same degree and same roots then implies that there exists some $\mu\in \K$ such that $f_{c,D}=\mu f_{a,B}$. Comparing the coefficients of $x$, we have that $\mu a^q=b^q$. However, comparing the coefficients of $x^{q^{t+1}}$ gives $\mu a=b$. Hence $\mu a^q=(\mu a)^q$ which implies that $\mu^q-\mu=0$ and $\mu\in\F_q$. Setting $\lambda=\mu$ we then have that $c=\lambda a$ and $D=\lambda B$. Conversely, suppose that $c=\lambda a$ and $D=\lambda B$ for some $\lambda\in\F_q$. Then
 \begin{align*} 
  f_{c,D}&=  (\lambda a)^q x+\lambda b_1 x^q+\cdots+\lambda b_t x^{q^t}+(-1)^{t+1} \lambda a x^{q^{t+1}}\\
  &=  \lambda a^q x+\lambda b_1 x^q+\cdots+\lambda b_t x^{q^t}+(-1)^{t+1} \lambda a x^{q^{t+1}}\\
  &=  \lambda f_{a,B}.
 \end{align*}
 Thus the result holds.
\end{proof}

The next result shows that every $t$-dimensional subspace of $\K$ is contained in at most one element of $\B_a$.

\begin{lemma}\label{lem:intersection}
 Let $a\in \K\setminus\{0\}$ and $\B_a$ be as in (\ref{eqn:Bsuba}). Then every element of $\B_a$ is a $(t+1)$-space of $\K$ and any pair of distinct elements intersect in at most a $(t-1)$-space. 
\end{lemma}

\begin{proof}
 Let $B=(b_1,\ldots,b_t)\in \K^t$. Then, since the coefficients of $x$ and $x^{q^{t+1}}$ in $f_{a,B}$ are both non-zero, it follows that $\ker f_{a,B}\in\B_a$ has dimension $t+1$. Let $C=(c_1,\ldots,c_t)\in \K^t$ with $B\neq C$. Then 
 \begin{align*}
  f_{a,B}-f_{a,C}=&(b_1-c_1) x^q+\cdots+(b_t-c_t) x^{q^t}\\
  =&\left(d_1 x+\cdots+d_{t} x^{q^{t-1}}\right)^q,
 \end{align*}
 where $d_i\in \K$ such that $d_i^q=b_i-c_i$, for $i=1,\ldots,t$. Moreover, since $B\neq C$, it follows that the greatest common divisor of $f_{a,B}$ and $f_{a,C}$ in $\K[x]$ divides $f_{a,B}-f_{a,C}$, and thus has at most $q^{t-1}$ distinct roots. Hence the kernels of $f_{a,B}$ and $f_{a,C}$ intersect in at most a $(t-1)$-subspace.
\end{proof}

Given the lemma above, the next result is enough for us to deduce that $\B_a$ is an $S(t,t+1,\K)_q$.

\begin{lemma}\label{lem:covered}
 Let $a\in \K\setminus\{0\}$ and $\B_a$ be as in (\ref{eqn:Bsuba}). Then every $t$-space of $\K$ is contained in some element of $\B_a$.
\end{lemma}

\begin{proof}
 Let $V=\langle v_1,\ldots,v_t\rangle$ be an arbitrary $t$-space of $\K$. Let $g\in \K[x]$ be defined by 
 \[
 g=
  \begin{vmatrix}
   v_1 & v_1^q & v_1^{q^2} & \cdots & v_1^{q^t} \\
   v_2 & v_2^q & v_2^{q^2} & \cdots & v_2^{q^t} \\
   \vdots & \vdots & \vdots & \ddots & \vdots\\
   v_t & v_t^q & v_t^{q^2} & \cdots & v_t^{q^t} \\
   x & x^q & x^{q^2} & \cdots & x^{q^t}
  \end{vmatrix}.
 \]
 If $x\in V$ then the bottom row in the determinant above is a linear combination of those above, which implies that $g(x)=0$. Since $g$ has degree $q^t$, it follows that $V$ is the set of all roots of $g$. Since $\K$ is algebraically closed and $a\neq 0$, it follows that there exists some $v_{t+1}\in \K\setminus V$ such that $g(v_{t+1})=a$. Thus, setting
 \[
  f_{a,B}=
  \begin{vmatrix}
   x & x^q & x^{q^2} & \cdots & x^{q^{t+1}}\\
   v_1 & v_1^q & v_1^{q^2} & \cdots & v_1^{q^{t+1}} \\
   v_2 & v_2^q & v_2^{q^2} & \cdots & v_2^{q^{t+1}} \\
   \vdots & \vdots & \vdots & \ddots & \vdots\\
   v_{t+1} & v_{t+1}^q & v_{t+1}^{q^2} & \cdots & v_{t+1}^{q^{t+1}}
  \end{vmatrix},
 \]
 defines $B\in \K^t$ such that $V\leq\ker f_{a,B}$. Since $\ker f_{a,B}\in\B_a$, the result holds.
\end{proof}

We are now in a position to prove the main theorem.


\begin{proof}[Proof of Theorem~\ref{thm:mainresult}]
 If $a\in \K\setminus\{0\}$ then it follows from Lemmas~\ref{lem:intersection} and~\ref{lem:covered} that $\B_a$ is an $S(t,t+1,\K)_q$. To see that $\LL=\{\B_a\mid a\in \K\setminus\{0\}\}$ is an $LS(t,t+1,\K)_q$ it remains to prove that for every $(t+1)$-space $V$ of $\K$ there exists some $b\in \K\setminus\{0\}$ such that $V$ lies in $\B_b$. Let $v_1,\ldots,v_{t+1}$ be a basis for $V$. Then setting
 \[
  f_{b,C}=
  \begin{vmatrix}
   x & x^q & x^{q^2} & \cdots & x^{q^{t+1}}\\
   v_1 & v_1^q & v_1^{q^2} & \cdots & v_1^{q^{t+1}} \\
   v_2 & v_2^q & v_2^{q^2} & \cdots & v_2^{q^{t+1}} \\
   \vdots & \vdots & \vdots & \ddots & \vdots\\
   v_{t+1} & v_{t+1}^q & v_{t+1}^{q^2} & \cdots & v_{t+1}^{q^{t+1}}
  \end{vmatrix},
 \]
 determines $b\in\K\setminus\{0\}$ and $C\in\K^t$ such that $V=\ker f_{b,C}\in\B_b$. This completes the proof.
\end{proof}

We finish with an open problem.

\begin{question}
 Can similar methods be used to construct an $LS(t,k,\K)_q$ for positive integers $t$ and $k$ such that $k\geq t+2$?
\end{question}


\end{document}